\documentclass{amsart}
\usepackage{amssymb}
\usepackage{amsthm}
\usepackage{mathtools}
\usepackage{environ}
\usepackage{enumerate}
\usepackage{microtype}
\usepackage{enumitem}
\usepackage{mathrsfs}
\usepackage[utf8]{inputenc}

\theoremstyle{plain}
\newtheorem{theorem}{Theorem}[section]
\newtheorem{lemma}[theorem]{Lemma}
\newtheorem{proposition}[theorem]{Proposition}
\newtheorem{corollary}[theorem]{Corollary}

\theoremstyle{definition}
\newtheorem{definition}[theorem]{Definition}

\theoremstyle{remark}
\newtheorem*{remark}{Remark}

\DeclareFontFamily{U}{mathx}{\hyphenchar\font45}
\DeclareFontShape{U}{mathx}{m}{n}{<-> mathx10}{}
\DeclareSymbolFont{mathx}{U}{mathx}{m}{n}
\DeclareMathAccent{\widebar}{0}{mathx}{"73}

\newcommand{\bd}{\partial}

\newcommand{\C}{\mathbb{C}}
\newcommand{\D}{\mathbb{D}}
\newcommand{\R}{\mathbb{R}}
\newcommand{\N}{\mathbb{N}}
\newcommand{\psh}{\mathcal{PSH}}

\newcommand{\suchthat}{\mathrel{;}}

\title[Quasibounded solutions to the complex Monge--Ampère equation]{Quasibounded solutions to the complex Monge--Ampère equation}
\author{Mårten Nilsson}
\thanks{This work was supported by a research grant from the Sverker Lerheden Foundation.}
\address{Department of Mathematics \\
Stockholm University \\
106 91 Stockholm, Sweden}
\email{marten.nilsson@math.su.se}

\subjclass[2010]{Primary 32U05; Secondary 32U15, 31C10, 31C05}

\begin{document}

\begin{abstract}
 We study the Dirichlet problem for the complex Monge--Ampère operator on a B-regular domain $\Omega$, allowing boundary data that is singular or unbounded. We introduce the concept of \textit{pluri-quasibounded} functions on $\Omega$ and $\partial \Omega$, defined by the existence of plurisuperharmonic majorants that dominate their absolute value in a strong sense---that is, the ratio of the function to the majorant tends to zero as the function tends to infinity. For such data, we prove existence and uniqueness of solutions in the Błocki--Cegrell class $\mathcal{D}(\Omega)$, using a recently established comparison principle. In the unit disk, our approach recovers harmonic functions represented as Poisson integrals of $L^1$ boundary data with respect to harmonic measure, and our characterization extends to all regular domains in $\R^n$, when the boundary data is continuous almost everywhere. We also describe how boundary singularities propagate into the interior via a refined pluripolar hull. 
\end{abstract}
\maketitle
\section{Introduction}
The complex Monge--Ampère equation plays a fundamental role in pluripotential theory, complex analysis, and complex geometry. While there has been substantial progress in understanding the Dirichlet problem in various settings—particularly when the boundary data is continuous—much less is understood in the presence of singularities on the boundary. In this work, we develop a framework for solving Dirichlet-type problems for the complex Monge–Ampère operator on B-regular domains, under general and potentially unbounded boundary conditions.

We are in particular interested in the case where the boundary data are allowed to be singular, but with control in the classical $L^1$ sense with respect to harmonic measure. Motivated by earlier results on quasibounded harmonic and plurisubharmonic functions \cite{arsove, nilsson, nilsson3, parreau,yamashita}, we introduce and study the class of \textit{(pluri-)quasibounded} boundary functions: functions whose absolute value is controlled, in a limiting sense, by a (pluri)superharmonic function defined on the interior of the domain. This notion turns out to be well suited for capturing integrability near the singularities, while ensuring the existence and uniqueness of maximal solutions, and for the inhomogeneous case, solutions in the Błocki–Cegrell class $\mathcal{D}(\Omega)$ (the canonical domain for the complex Monge–Ampère operator \cite{blocki}).

In the first part of the paper, we revisit the classical Dirichlet problem for harmonic functions on the unit disk $\D$, providing a motivating case for our approach. We show that a harmonic function $u$ arises as the Poisson integral of an $L^1$-boundary function if and only if it is \textit{quasibounded}, meaning that there exists a positive superharmonic function $v$ such that
\[
\frac{v(z)}{|u(z)|} \to \infty \quad \text{as } |u(z)| \to \infty.
\]
 An equivalent rephrasing is that $u$ and $v$ relate in the following way: For each $\varepsilon >0$, there exists $M_\varepsilon$ such that
\[
 |u| \leq \varepsilon v + M_\varepsilon \quad \text{in }\D.
\]
This perspective allows one to characterize boundary integrability in terms of potential-theoretic behavior within the domain, rather than through direct reference to harmonic measure. More precisely, given a boundary function $\varphi$ taking values in the extended reals or in the extended complex plane, we say that $\varphi$ is quasibounded if there exists a positive superharmonic function $v$ such that for each $\varepsilon >0$, there exists $M_\varepsilon$ such that
    \[
    |\varphi(\zeta)|   \leq \liminf_{z \rightarrow \zeta} \varepsilon v(z) + M_\varepsilon, \quad \forall \zeta\in\partial \D,
    \]
yielding the equivalence
\[
\varphi \in L^1(\partial \D) \iff \varphi \text{ quasibounded.}
\]
We also show that this characterization of integrability with respect to harmonic measure generalizes to every \(\Omega \subset \C_\infty\) with non-polar boundary (under the convention that $\infty \in \partial \Omega$ if $\Omega$ is unbounded) and to every regular domain in $\R^n$, under the assumption that $\varphi$ is continuous and locally bounded outside a set of harmonic measure 0. Finally, by identifying sets of harmonic measure 0 with the b-polar sets, we remove all dependence on harmonic measure. Recall that a set $A \subset \partial \Omega$ is said to be b-(pluri)polar if there exists a negative (pluri)subharmonic function $u$, not identically $-\infty$, such that 
\[
\limsup_{\Omega\ni z\rightarrow A}u(z) =-\infty.
\]

The main results of the paper concern the complex Monge–Ampère equation in higher dimensions. In Section 3, we show that if $\Omega \subset \mathbb{C}^n$ is a B-regular domain and $\varphi : \partial \Omega \to \mathbb{R} \cup \{ \pm\infty \}$ is \textit{pluri-quasibounded} (quasibounded by a plurisuperharmonic function), continuous and locally bounded outside a b-pluripolar set $E_\varphi$, then there exists a unique pluri-quasibounded solution to the Dirichlet problem
    \[
    \begin{cases}
    u \in {\mathcal{D}}(\Omega) \\
    (dd^cu)^n=\mu \\
        \lim_{w \rightarrow z}u(w) = \varphi(z) \qquad \forall z \in \partial \Omega \setminus  E_\varphi,
    \end{cases}
    \]
provided $\varphi$ admits a minorant in the Cegrell class $\mathcal{E}(\Omega)$ and $\mu$ is a compliant measure. The method of proof relies on a combination of envelope techniques, monotonicity arguments, and critically on a comparison principle in Cegrell classes due to Åhag--Cegrell--Czyż--Hiep~\cite{ahag2} (which was recently improved to also encompass singular measures in \cite{ahag}).

In Section 4, we investigate how boundary singularities affect interior regularity. We define a refined version of the pluripolar hull and show that the solution $u$ is continuous outside the closure of the b-pluripolar hull of the boundary singularities, which reflects the propagation of boundary singularities into the interior. The characterization is optimal, since the discontinuity set may coincide with the closure of the b-pluripolar hull already when the boundary function is bounded \cite{nilsson2}. 

Finally, Section 5 treats the homogeneous case, replacing $(dd^c u)^n = 0$ with the more general notion of maximality. We show that under the same assumptions on the boundary data, a unique pluri-quasibounded, maximal solution exists if and only if the singularity set is b-pluripolar. Unlike the inhomogeneous case, this result does not require the existence of a minorant in $\mathcal{E}(\Omega)$, making it applicable in strictly broader settings (see the first remark to Theorem~\ref{mainresult}).

The author would like to thank Alan Sola for many valuable comments on earlier versions of the manuscript.

\section{Quasibounded harmonic functions on domains in the plane}
Given a harmonic function $u$ defined on the unit disk, we shall prove that there exists a positive superharmonic function $v$ with the property that
    \begin{equation}\label{quasi1}
      \frac{v(z)}{|u(z)|} \rightarrow \infty \text{ as }|u(z)| \rightarrow \infty, \tag{$*$}   
    \end{equation}
     if and only if $u$ can be represented by an element in $L^1(\partial \D)$. The tools necessary are due to Parreau~\cite{parreau}, who proved (or at least remarked on) a similar result for positive harmonic functions $u$ for which $\phi \circ u$ has a harmonic majorant, where $\phi:\R\to\R$ is increasing, convex and
\[
\lim_{t\rightarrow \infty} \frac{\phi(t)}{t} = \infty.
\]
The essence of the argument is that such harmonic functions are \textit{quasibounded} in the following sense: $u$ is quasibounded if there exists a sequence of bounded real-valued harmonic functions $u_n$ such that $u_n\nearrow u$. For (unbounded) complex-valued functions, we instead make the following definition. 
\begin{definition}
We say that $u$ is \textit{quasibounded} if there exists a positive superharmonic function $v$ such that (\ref{quasi1}) holds.
\end{definition}
The following lemma (see also \cite[Theorem~4.1]{arsove}) shows that the two notions coincide for real-valued harmonic functions bounded from below.
\begin{lemma}\label{karaktarlemma}
    Suppose that $u$ is harmonic, real-valued, and bounded from below. Then $u$ is quasibounded if and only if there exists a sequence of bounded harmonic functions $u_k$ such that $u_k\nearrow u$.
\end{lemma}
\begin{proof}
Suppose that $u>K$, and note that if $u$ is quasibounded by a superharmonic function $v$, then 
    \[
      \frac{v(z)}{k|u(z)|} \rightarrow \infty \text{ as }|u(z)| \rightarrow \infty
    \]
holds for all $k \in \N$. In particular, $\max(u - \frac{v}{k},K)$ is a bounded subharmonic function, and taking the lower envelope of all of its superharmonic majorants produces bounded harmonic functions $u_k$, with $u_k\nearrow u$ as $k \rightarrow \infty$.

For the other direction, fix a point $z_0$ and choose a subsequence such that $u(z_0) - u_k(z_0) \leq \frac{1}{2^k}$. Then 
\[
v = \sum^{\infty}_{k=1} (u - u_k)
\]
is a (super)harmonic function by Harnack's theorem. To see that $v$ has the desired properties, let $n_k$ be an increasing sequence of integers such that $u_k \leq n_k$ and note that $\max(u-n_k,0)$ is majorized by the harmonic function $u - u_k$. Now define 
\[
\varphi(t) := \sum^{\infty}_{k=1}\max(t-n_k, 0) \quad -\infty < t < +\infty.
\]
This series reduces to a finite sum whenever $t<+\infty$, and is piecewise linear with slope $k$ on the interval $[n_k, n_{k+1}]$, and therefore increasing and convex. Hence 
\[
      \frac{\varphi(u(z))}{u(z)} \rightarrow \infty \text{ as }u(z) \rightarrow \infty,
    \]
and since $\varphi \circ u \leq v$, the claim follows.
\end{proof}
We are now ready to prove
\begin{theorem}\label{karaktar_l1}
A harmonic function $u$ defined on the unit disk is quasibounded if and only if it is represented by an element in $L^1(\partial \D)$.
\end{theorem} 
\begin{proof}
Suppose that $u$ is represented by $\varphi \in L^1(\partial \D)$, and write 
\[
\varphi=\varphi_1 - \varphi_2 +i(\varphi_3 - \varphi_4)
\]
with each $\varphi_i \geq 0$.  Then the bounded harmonic functions 
\[
u^{(i)}_n(z) := \int_0^{2\pi}   P_z(\theta) \min(\varphi_i(\theta), n) \,\frac{d\theta}{2\pi}
\]
approximate the Poisson integral $u^{(i)}$ of $\varphi_i$ from below, and by the lemma, there exist (super)harmonic functions $v_i$ such that 
    \[
      \frac{v_i(z)}{u^{(i)}(z)} \rightarrow \infty \text{ as }u^{(i)}(z) \rightarrow \infty,
    \]
since $u^{(i)}\geq 0$. Now let $z_k$ be a sequence such that $|u(z_k)| \rightarrow \infty$, and let $n_k \in \{1,2,3,4\}$ be such that $u^{(n_k)}(z_k)=\max(u^{(i)}(z_k))$. It follows that $u^{(n_k)}(z_k)\rightarrow \infty$ as $k\rightarrow \infty$, and taking $v = \sum v_i$, we have 
    \[
      \frac{v(z_k)}{|u(z_k)|} \geq \frac{v_{n_k}(z_k)}{4 u^{(n_k)}(z_k)}  \rightarrow \infty,
    \]
which shows that $u$ is quasibounded.

Conversely, suppose that there exists a positive superharmonic function $v$ such that
    \[
      \frac{v(z)}{|u(z)|} \rightarrow \infty \text{ as }|u(z)| \rightarrow \infty.
    \]
It is enough to show that the real part of $u$ is represented by an element in $L^1(\partial \D)$. Since $|\operatorname{Re} u| \leq |u|$, we note that $\operatorname{Re} u$ is also quasibounded, and that we for any $\varepsilon>0$ may find $K_\varepsilon>0$ such that 
\[
\operatorname{Re} u < \varepsilon v  + K_\varepsilon.
\]
Hence, the least harmonic majorant $\tilde u$ to the subharmonic function $\max(\operatorname{Re} u, 0)$ satisfies
\[
0\leq  \max(\operatorname{Re} u, 0) \leq \tilde u < \varepsilon v  + K_\varepsilon,
\]
and so
    \[
      \frac{v(z)}{|\tilde u(z)|}\geq \frac{v(z)}{\varepsilon v(z)  + K_\varepsilon}  \rightarrow \frac{1}{\varepsilon} \quad\text{ as }|\tilde u(z)| \rightarrow \infty.
    \]
We conclude that $\tilde u$ is quasibounded as well. By the lemma, there exists a sequence of bounded harmonic functions $u_n \nearrow \tilde u$, all represented by elements 
\[
\varphi_n \in L^{\infty}(\partial \D) \subset L^1(\partial \D),
\]
and by the monotone convergence theorem, there exists an element $\varphi \in L^1(\partial \D)$ such that 
\[
\tilde u(z)  = \int_0^{2\pi} \varphi(e^{i\theta}) \frac{1-|z|^2}{|e^{i\theta}-z|^2}\,\frac{d\theta}{2\pi}.
\]
Similarly, as the difference $\tilde u - \operatorname{Re} u$ is non-negative, harmonic and quasibounded, $\tilde u - \operatorname{Re} u$ and $\operatorname{Re} u$ are also represented by elements in $L^1(\partial \D)$, which finishes the proof.
\end{proof}
\begin{remark}
Note that these proofs also imply that $u$ is represented by an element in $L^1(\partial \D)$ if and only if there exists a positive \textit{harmonic} function $v$ such that 
    \[
      \frac{v(z)}{|u(z)|} \rightarrow \infty \text{ as }|u(z)| \rightarrow \infty,
    \]
where $v$ is quasibounded as well. We should also remark (see also \cite[Theorem~2]{yamashita}) that this reduces the F. and M. Riesz theorem to the following statement:
\[
\textit{$f$ analytic with $|f|\leq h$ for some harmonic function $h$}\implies \textit{$f$ quasibounded.}
\]
This implication is a consequence of the fact that $\max(\log |f|, 0)$ is subharmonic and $\log|h+1|$ is superharmonic, so we may find a harmonic function $h'$ such that
\[
\max(\log |f|, 0) \leq h' \leq \log|h+1|,
\]
which is quasibounded since
\[
\frac{h(z)}{h'(z)} \geq \frac{h(z)}{\log|h(z) + 1|} \rightarrow \infty \text{ as }|h'(z)| \rightarrow \infty.
\]
By the lemma, we may find bounded harmonic functions $h'_n$ such that $h'_n \nearrow h'$. Now let $h''_n$ be the smallest harmonic majorant to $e^{h'_n}$. Since this increasing sequence of harmonic functions is bounded from above by $h+1$, we have $h''_n \nearrow h''$, where $h''$ is harmonic by Harnack's theorem. We conclude that $h''$ is quasibounded, and therefore $f$ is as well. 
\end{remark}
The main benefit of Theorem~\ref{karaktar_l1} is that it allows us to consider the Dirichlet problem with unbounded boundary data, without any reference to harmonic measure on the circle. Indeed, it implies that for any quasibounded \(\varphi: \partial \D \rightarrow \C_\infty \), continuous and locally bounded outside a b-polar set, there exists a unique quasibounded harmonic function \(h\) defined on $\D$ with
\[\lim_{z \rightarrow \zeta \in \partial \D}h(z) = \varphi(\zeta) \quad\text{a.e.,}\] 
since b-polar sets are precisely those of zero arc length (see \cite{nilsson2}). 

As mentioned in the Introduction, this viewpoint extends to all domains with non-polar boundary. 
\begin{theorem}\label{dirichlet_endim}
 Let \(\Omega \subset \C_\infty\) be a domain with non-polar boundary and suppose that \(\varphi: \partial \Omega \rightarrow \R \cup \{-\infty, \infty\} \) is quasibounded, continuous and locally bounded outside $E_\varphi$. Then there exists a unique quasibounded harmonic function \(h\) defined on $\D$ such that 
\[\lim_{z \rightarrow \zeta \in \partial \Omega \setminus E_\varphi}h(z) = \varphi(\zeta)\quad\text{outside a b-polar set}\] 
if and only if $E_\varphi$ is b-polar.
\end{theorem}
\begin{proof}
 Suppose $E_\varphi$ is b-polar. The proof of \cite[Corollary~4.2.6]{ransford}, slightly modified by replacing polar sets with b-polar sets, combined with \cite[Theorem~4.3.3]{ransford} implies that all bounded harmonic functions with boundary limits outside a b-polar set are representable as integrals with respect to harmonic measure. Hence, for any $n \in \N$, we may solve the Dirichlet problem uniquely for the boundary functions $\min(n,\max(\varphi, 0))$ and $\max(-n,\min(\varphi, 0))$. Letting $n \rightarrow \infty$, we acquire quasibounded harmonic functions $h_+, h_-$ from Harnack's theorem. Clearly, the sum $h_+ + h_-$ has the correct boundary behavior, and by monotone convergence, is represented by an integral.
 
 Now suppose that $h$ is another solution, and let $\tilde h$ be the smallest harmonic majorant to $\max(h,0)$. Since $\tilde h$ is quasibounded, there exist bounded harmonic functions such that $h_n \nearrow \tilde h$. Let $C_n$ be an increasing sequence of constants such that
 \[\limsup_{z \rightarrow \zeta \in \partial \Omega \setminus E_\varphi}h_n(z) \leq \min(C_n,\max(\varphi(\zeta), 0))\quad\text{outside a b-polar set,}\] 
and let $\tilde h_n$ denote the unique bounded harmonic function with the right-hand side as its boundary limits. By the extended maximum principle, $h_n \leq \tilde h_n \leq \tilde h$. Since $\tilde h_n$ is representable as an integral with respect to harmonic measure, we may reason as in the sufficiency part of the proof of Theorem~\ref{karaktar_l1} to conclude that $\tilde h$ and $h$ are also represented by integrals. The solution is therefore unique.

Conversely, suppose that $E_\varphi$ is not b-polar, and let $\chi_{E_\varphi}$ denote the indicator function on $E_\varphi$. Then $E_\varphi$ does not have zero harmonic measure, since for every $z \in \Omega, k\in \N$, the harmonic measure $\omega_\Omega$ satisfies
\begin{align*}
-\omega_\Omega(z, E_\varphi)&=\int_{\partial\Omega} - \chi_{E_\varphi}(\zeta) \ d\omega_\Omega(z, \zeta) \\
&= \sup\{u(z)\suchthat u \text{ subharmonic, }u^*\leq - \chi_{E_\varphi} \text{ on }\partial\Omega\}
\end{align*}
by \cite[Theorem~4.3.3]{ransford}. The right-hand side cannot be equal to zero, since otherwise we could find subharmonic functions $u_k$ with the properties that
\[
u_k^* \leq -\chi_{E_\varphi} \text{ on }\partial\Omega, 
 \quad u_k(z)>-\frac{1}{2^k}, 
\]
from which we may construct a negative subharmonic function
\[
\tilde u := \sum_{k=1}^\infty u_{k}
\]
with $\tilde u^* \leq -\infty $ on $E_\varphi$, implying that $E_\varphi$ is b-polar. Therefore, adding the Poisson integral of $\chi_{E_\varphi}$ to one solution produces another.
\end{proof}
\begin{remark}
Note that this theorem is the most general version of Corollary~4.2.6 in \cite{ransford}.
\end{remark}
Omitting the details, it is clear that one can use corresponding notions of potential theory in $\R^n$ to prove the following result, which we shall use later on. 
\begin{corollary}\label{dirichlet_flerdim}
 Let \(\Omega \subset \R^n\) be a regular domain and suppose that \(\varphi: \partial \Omega \rightarrow \R \cup \{-\infty, \infty\} \) is quasibounded, continuous and locally bounded outside $E_\varphi$. Then there exists a unique quasibounded harmonic function \(h\) such that 
\[\lim_{z \rightarrow \zeta \in \partial \Omega \setminus E_\varphi}h(z) = \varphi(\zeta)\] 
if and only if $E_\varphi$ is b-polar.
\end{corollary}

\section{Quasibounded plurisubharmonic functions}
For the higher dimensional case, we shall assume that the domain $\Omega \subset \C^n$ is B-regular, and restrict our attention to functions taking values in the extended reals. 
\begin{definition}
We say that a function $f:\Omega \rightarrow \R \cup \{-\infty, \infty\}$ is \textit{pluri-quasibounded} if there exists a positive plurisuperharmonic $v$ defined on $\Omega$ such that
    \[
      |f(z_0)|=+\infty \implies v(z_0)=+\infty \quad \text{and} \quad
      \frac{v(z)}{|f(z)|} \rightarrow \infty \text{ as }|f(z)| \rightarrow \infty.
    \]
Equivalently, for every $\varepsilon >0$, there exists $M_\varepsilon$ such that
\[
-\varepsilon v - M_\varepsilon \leq f \leq  \varepsilon v + M_\varepsilon,
\]
and we say that $v$ (and $-v$) quasibounds $f$ from above (resp. from below).
\end{definition}
\begin{remark}
A similar notion may be found in \cite{nilsson}. As shown in \cite[Lemma~2.1]{nilsson}, a third characterization is that we may find $\psi:[0,+\infty] \rightarrow[0,+\infty]$ satisfying
\[
\lim_{t\rightarrow+\infty}\frac{\psi(t)}{t}=0,
\]
such that $|f|\leq\psi(v)$ for some positive plurisuperharmonic function $v$ defined on $\Omega$. Although this is closer to Parreau's original theory, its main utility is that it provides many examples of pluri-quasibounded functions, and we shall not make use of it in our arguments. A fourth equivalent condition, for plurisubharmonic functions bounded from below, is that there exists an increasing sequence of bounded plurisubharmonic functions $u_n\leq u$ such that $(u_n-u)^*$ is plurisubharmonic and $u_n-u \nearrow 0$ outside a pluripolar set \cite[Theorem~3.3]{nilsson}. Recall that
\[
u^*(z) := \limsup_{\Omega \ni w\rightarrow z} u(w) \qquad u_*(z) := \liminf_{\Omega \ni w\rightarrow z} u(w)
\]
denote the upper (resp. lower) semicontinuous regularization of $u$.
\end{remark}
 
 As in the single variable case, we say that $\varphi: \partial \Omega \rightarrow \R \cup \{-\infty, \infty\}$ is pluri-quasibounded if there exists a positive plurisuperharmonic function  $v$ with the property that for all $\varepsilon >0$, there exists $M_\varepsilon$ such that
 \[
 |\varphi(\zeta)|   \leq \liminf_{z \rightarrow \zeta} \varepsilon v(z) + M_\varepsilon, \quad \forall \zeta \in \partial\Omega.
\]


 In order to arrive at a pluricomplex counterpart to Theorem~\ref{dirichlet_endim}, we will need some results that ensure that all complex Monge--Ampère measures we come across are well-defined. Effectively, this will be a matter of showing that all negative plurisubharmonic functions we invoke may be replaced by functions in the Cegrell class $\mathcal{E}(\Omega)$, which requires that $\Omega$ is at least hyperconvex. The class $\mathcal{E}(\Omega)$ is characterized by the following property \cite[Theorem~4.2]{cegrell2}: It is the largest class for which the complex Monge--Ampère operator $u \mapsto (dd^cu)^n$ is well-behaved, in the sense that any other $\mathcal{K}(\Omega) \subset \psh^-(\Omega)$ such that
\begin{enumerate}[label=(\roman*)]
    \item $u \in \mathcal{K}(\Omega), v\in\psh^-(\Omega) \implies \max(u,v)\in \mathcal{K}(\Omega)$,
    \item If $u \in \mathcal{K}(\Omega), v_j\in\psh^-(\Omega)\cap L^\infty(\Omega)$ such that $v_j \searrow u$,  then $\big((dd^cv_j)^n\big)_{j=1}^\infty$ is weak*-convergent
\end{enumerate}
necessarily satisfies $\mathcal{K}(\Omega) \subset \mathcal{E}(\Omega)$. More generally, it is known that the complex Monge--Ampère operator is continuous with respect to monotone sequences in $\mathcal{E}(\Omega)$. Concretely, $u \in \mathcal{E}(\Omega)$ if for each $z_0 \in \Omega$, there exists a neighborhood $U_{z_0} \subset \Omega$ and decreasing sequence of bounded functions $h_j$ such that $\lim_{z \rightarrow \zeta_0 \in \partial \Omega}h_j(z)=0,$ 
\[
 \sup_j \int (dd^c h_j)^n < \infty,
\]
and $h_j \searrow u$ on $U_{z_0}$. If we may choose $U_{z_0} = \Omega$, then $u \in \mathcal{F}(\Omega)$; this is a subclass of $\mathcal{E}(\Omega)$ enjoying many additional properties, which we will exploit in the proofs below. Finally, let $\mathcal{K}^a(\Omega)$ denote the subset of functions in $\mathcal{K}(\Omega)\subset\mathcal{E}(\Omega)$ whose Monge--Ampère measure does not charge pluripolar sets.  

We are now ready to state and prove the following two lemmata.

\begin{lemma}\label{Equasi}
Suppose that $u \in \mathcal{E}^a(\Omega)$ is quasibounded by $v$ from below. Then there exists $\tilde v \in \mathcal{E}(\Omega)$ that quasibounds $u$ from below.
\end{lemma} 
\begin{proof}
Let
\[
\tilde v := \sum_{l=1}^\infty \max\big(\frac{1}{2^{k_l}}v, u\big),
\]
and let $(z_j)$ be a sequence such that $u(z_j) \rightarrow -\infty$. Then $\frac{v(z_j)}{u(z_j)} \rightarrow \infty$, so 
\[
\liminf_{j} \frac{\max(\frac{1}{2^{k_l}}v(z_j), u(z_j))}{u(z_j)} \geq 1, 
\]
which implies that $\tilde v$ quasibounds $u$. It remains to show that we may choose $k_l$ so that $\tilde v \in \mathcal{E}(\Omega)$. Since $\max(\frac{1}{2^{k_l}}v, u) \in \mathcal{E}^a(\Omega)$, by \cite[last remark on p.~166]{cegrell2}, for each open set $\Omega' \Subset \Omega$ there exists an element in $\mathcal{F}(\Omega)$ coinciding with $\max(\frac{1}{2^{k_l}}v, u)$ on $\Omega'$. Therefore, for subdomains $\Omega_{1}\Subset\Omega_2\Subset \dots\subset \Omega$ such that $\cup_l \Omega_l = \Omega$, the functions
\[
v_l(z) := \sup\{u(z) \in \psh^-(\Omega) \suchthat u \leq \max(\frac{1}{2^{k_l}}v, u) \text{ on }\Omega_l\}^*
\]
belongs to $\mathcal{F}^a(\Omega)$,  and by balayage (here we use that $u\in \mathcal{E}^a(\Omega)$, see \cite[Theorem~2.7]{ahag}), $(dd^cv_l)^n = 0$ on $\Omega \setminus \bar \Omega_l$. Clearly, as $k_l \rightarrow \infty$, $v_l \nearrow 0$ outside a pluripolar set, so by the fact that the complex Monge--Ampère operator is continuous on increasing sequences in $\mathcal{E}(\Omega)$, $(dd^cv_l)^n$ tends weakly to $0$. This permits us to choose $k_l$ such that
\[
\int_\Omega (dd^c v_l)^n \leq \frac{1}{2^l}.
\]
Now let $\Omega' \Subset \Omega$ be an arbitrary open set. By compactness, there exists $l_0$ such that $\Omega' \subset \Omega_{l_0}$, and using \cite[Corollary 5.6]{cegrell2}, we have that
\[
\Big(\int_\Omega \big(dd^c \sum_{l=l_0}^N v_{l}\big)^n\Big)^{1/n} \leq \sum_{l=l_0}^N \Big(\int_\Omega (dd^c v_{l})^n\Big)^{1/n} \leq \sum_{l=l_0}^N \frac{1}{2^{l/n}} < \frac{1}{1-\frac{1}{2^{1/n}}}.
\]
By \cite[Lemma~2.1]{czyz}, the partial sums converge to an element in $\mathcal{F}(\Omega)$. But $\mathcal{E}(\Omega)$ is a cone, so
\[
\sum_{l=1}^{l_0-1} \max(\frac{1}{2^{k_l}}v, u)+\sum_{l=l_0}^\infty v_{l} \in \mathcal{E}(\Omega),
\]
coinciding with $\tilde v$ on $\Omega'$. This shows that $\tilde v \in \mathcal{E}(\Omega)$. 
\end{proof}
\begin{remark}
A result of Rashkovskii (see \cite[Corollary~3.4]{rashkovskii} and \cite[Corollary~3.5]{nilsson}) shows that every element in $\mathcal{E}^a(\Omega)$ is locally pluri-quasibounded, in the sense that their restriction to any hyperconvex subdomain $\Omega'\Subset \Omega$ is pluri-quasibounded on $\Omega'$.
\end{remark}
With a similar proof, we also obtain
\begin{lemma}\label{bpluripolar}
Suppose $A \subset \partial \Omega$ is b-pluripolar. Then there exists $u \in \mathcal{E}(\Omega)$ such that $\limsup_{z\rightarrow A} u(z) = -\infty$.
\end{lemma} 
\begin{proof}
Since $A$ is b-pluripolar, the boundary relative extremal function
\[
\omega^*(z,A,\Omega) := \big(\sup\{u(x) \suchthat u\in \psh^-(\Omega), u \leq -1 \text{ on }A\}\big)^*
\]
satisfies $\omega^*(z,A, \Omega) \equiv 0$ \cite[Proposition~3.5]{djire}. Hence we may find bounded plurisubharmonic functions $u_{k_j}\nearrow 0$ outside a pluripolar set such that $u_{k_j}|_A = -1$ and $u_{k_j}(z_0) \geq \frac{1}{2^{k_j}}$ at some point $z_0 \in \Omega$. By Harnack's theorem,
\[
u := \sum_{l=1}^\infty u_{k_l}
\]
converges to a plurisubharmonic function with $\limsup_{z\rightarrow A} u(z) = -\infty$. Exchanging $\max(\frac{1}{2^{k_l}}v, u)$ for $u_{k_l}$ in the proof of Lemma~\ref{Equasi}, the conclusion follows. 
\end{proof}
We will need two additional results concerning plurisubharmonic functions in $\mathcal{E}(\Omega)$. The following lemma is a weakened version of the comparison principle found in \cite{ahag}, on which our main result will depend in an essential way.
\begin{lemma}\label{jamfor}
Fix $v \in \mathcal{E}(\Omega)$, and suppose that $u$ is a plurisubharmonic function satisfying
\[
 v \geq u \geq \psi_u + v, 
\]
where $\psi_u\in \mathcal{F}(\Omega)$.  Then
\[
(dd^cu)^n \leq (dd^cv)^n \implies u = v.
\]
\end{lemma}
\begin{proof}
This follows directly from \cite[Theorem~3.5]{ahag}, since $u, v \in \mathcal{F}(v) \subset \mathcal{N}(v)$ (see \cite[Section~2]{ahag}), and $v \geq u \implies v \succeq u$ (for all $K \Subset \Omega$, there exists $C_k$ such that $v + C_k \geq u$ on $K$).
\end{proof}
\begin{remark}
    With the additional assumption that $u\in \mathcal{E}^a(\Omega)$,  which is sufficient for our purposes, this result may also be proved using \cite[Corollary 3.2]{ahag2}. 
\end{remark}
We shall also need the following result from \cite{hiep}.
\begin{lemma}\label{max}
Let $\mu$ be a positive measure which vanishes on all pluripolar subsets of $\Omega$. Suppose $u,v \in \mathcal{E}(\Omega)$ are such that $(dd^c u)^n \geq \mu,(dd^c v)^n \geq \mu$. Then $(dd^c \max(u, v))^n \geq \mu$.
\end{lemma}
As we will consider plurisubharmonic functions that are not necessarily negative, the Cegrell classes are not enough for our purposes. Instead, we will consider the following extension of $\mathcal{E}(\Omega)$, due to Błocki~\cite{blocki, blocki2}:
\begin{definition}
We say that $u\in\mathcal{D}(\Omega)\subset\psh(\Omega)$, and write $(dd^cu)^n =\mu$, if there exists a measure $\mu$ on $\Omega$ such that for every open set $U \subset \Omega$ and every sequence $u_j \in \psh \cap C^\infty(U)$ decreasing to $u$ in $U$, the sequence of measures $(dd^c u_j)^n$ converges weakly (i.e., against compactly supported test functions) to $\mu$ in $U$. 
\end{definition}
\begin{remark}
Błocki proved ~\cite[proof of Theorem~1.1]{blocki} that $\psh^-(\Omega)\cap\mathcal{D}(\Omega)=\mathcal{E}(\Omega)$ if $\Omega$ is hyperconvex. This yields a localization of $\mathcal{E}(\Omega)$, i.e., 
\[
u  \in \mathcal{E}(\Omega) \iff u \in \mathcal{E}(\Omega_i) \text{ for all }i.
\]
Also note that continuity of the Monge--Ampère operator with respect to monotone sequences is naturally inherited from $\mathcal{E}(\Omega)$. This follows as plurisubharmonic functions are always locally bounded from above, so for any hyperconvex subdomain $\Omega' \Subset \Omega$, $u_k \in \psh(\Omega)$, there exists $C_k>0$ such that
\[
u_k\in \mathcal{D}(\Omega) \implies u-C_k \in \mathcal{E}(\Omega').
\]
If the sequence $u_k$ converges monotonically to $u\in \mathcal{D}(\Omega)$, then we can choose the same constant for all $u_k$, and it follows that $(dd^cu_k)^n \rightarrow (dd^cu)^n$ weakly on $\Omega'$. We should also remark that the other defining property of $\mathcal{E}(\Omega)$ also holds in $\mathcal{D}(\Omega)$: If $u,v\in \psh(\Omega)$ then
\[
u \in \mathcal{D}(\Omega) \implies \max(u,v) \in \mathcal{D}(\Omega).
\]
This is a special case of \cite[Theorem~1.2]{blocki}.
\end{remark}
As we are mainly concerned with how the boundary data influence regularity, existence, and uniqueness of a solution, we will restrict the class of allowed measures to those that are \textit{compliant}, i.e. measures for which the Dirichlet problem with regard to continuous boundary values has a bounded solution. The existence of compliant measures is equivalent to the requirement that the domain be B-regular (see \cite{nilsson2}). Importantly, such measures cannot charge pluripolar sets. See \cite[Theorem~6]{xing} for a local characterization of compliant measures.

With the preliminaries in place, we now prove

\begin{theorem}\label{mainresult}
    Let $\mu$ be compliant, and suppose that $\varphi: \partial \Omega \rightarrow \R\cup\{-\infty, \infty\}$ is pluri-quasibounded, continuous and locally bounded outside a b-pluripolar set $E_\varphi$, with minorant in $\mathcal{E}(\Omega)$. Then the Dirichlet problem
    \[
    \begin{cases}
    u \in {\mathcal{D}}(\Omega) \\
    (dd^cu)^n=\mu \\
        \lim_{w \rightarrow z}u(w) = \varphi(z) \qquad \forall z \in \partial \Omega \setminus  E_\varphi,
    \end{cases}
    \]
    has a unique pluri-quasibounded solution. 
\end{theorem} 
\begin{proof}
Suppose first that $\varphi$ is bounded from above. Adding a constant if necessary, we may assume $\varphi\leq 0$. By assumption, there exists $v_\varphi \in \mathcal{E}(\Omega)$ such that
\[
\limsup_{w \rightarrow z}v_\varphi(w) \leq \varphi(z) \qquad \forall z \in \partial \Omega ,
\]
and adding an element in $\mathcal{E}(\Omega)$, we may suppose that $(dd^cv_\varphi)^n\geq\mu$.  Using \cite[Theorem~2.3]{nilsson2} there exist unique functions $u_k$ that solve \[
    \begin{cases}
    u_k \in \psh(\Omega)\cap L^\infty(\Omega)  \\
    (dd^cu_k)^n=\mu \\
        \lim_{w \rightarrow z}u_k(w) = \max(\varphi(z), -k) \qquad \forall z \in \partial \Omega \setminus  E_\varphi,
    \end{cases}
    \]
and by \cite[Corollary 3.2]{ahag2}, $u_k \geq v_\varphi$. Therefore, $u_k \searrow u \in \mathcal{E}^a(\Omega)$, and by the continuity of the complex Monge--Ampère operator on monotone sequences in $\mathcal{E}(\Omega)$, $u$ is a solution. 

We claim that any other solution $u'\in \mathcal{E}^a(\Omega)$ must coincide with $u$. Using that $\max(u_k, u') \leq u_k$ for all $k$ according to Lemma~\ref{max} and \cite[Lemma~2.1]{nilsson2}, we immediately deduce that $u \geq u'$. By Lemma~\ref{Equasi}, $u'$ is quasibounded by an element $v \in \mathcal{E}(\Omega)$, and using Lemma~\ref{bpluripolar}, we may suppose that $\limsup_{z\rightarrow E_\varphi} v(z) = -\infty$ (including points where $\varphi$ is bounded but discontinuous). Further, we may suppose that $v \leq C < 0$. Then, for each $\varepsilon>0$, the set
\[
K_\varepsilon := \overline{\{z \in \Omega \suchthat \max(u + \varepsilon v, u') > u'\}}
\]
will be compact. Now
\[
\max(u + \varepsilon v, u') \geq u' \geq \max(u + \varepsilon v, u')+ \psi_{u'},
\]
where $\psi_{u'} \in \mathcal{F}(\Omega)$ is such that $\psi_{u'} = u'$ in a precompact neighborhood of $K_\varepsilon$. Lemma~\ref{jamfor} implies that $\max(u + \varepsilon v, u') = u$, so $u + \varepsilon v\leq u' \leq u$. Letting $\varepsilon \rightarrow 0$, we conclude that $u'=u$.

In order to extend to unbounded $\varphi$, we adapt the proof of \cite[Theorem~3.2]{nilsson}. First note that by Corollary~\ref{dirichlet_flerdim}, there is a unique quasibounded harmonic function $h_\varphi$ with $\varphi$ as its boundary limits outside $E_\varphi$. By uniqueness, $h_\varphi$ is pluri-quasibounded, and so the envelope
\[
U := \sup\{u(z) \in  {\mathcal{D}}(\Omega) \suchthat u \leq h_\varphi, (dd^cu)^n\geq \mu\}
\]
is pluri-quasibounded as well, and will majorize any solution. Following standard procedures, we conclude that $U^* = U$, $(dd^cU)^n \geq \mu$, and that 
\[
\lim_{w \rightarrow z}U(w) = \varphi(z) \qquad \forall z \in \partial \Omega \setminus  E_\varphi.
\]
In order to show that $U$ is a solution, it remains to show that $(dd^cU)^n = \mu$. Let $v$ be a strictly positive plurisuperharmonic function that quasibounds $U.$ Then there exist $M_1 \leq M_2 \leq \dots$ such that
\[
|U| < \frac{1}{k} v +  M_k,
\]
hence $U-\frac{1}{k} v$ is bounded from above. Let $U_k$ denote the unique quasibounded plurisubharmonic functions that satisfy
    \[
    \begin{cases}
    U_k \in {\mathcal{D}}(\Omega) \\
    (dd^cU_k)^n=\mu \\
        \lim_{w \rightarrow z}U_k(w) = \min(\varphi(z), M_k) \qquad \forall z \in \partial \Omega \setminus  E_\varphi.
    \end{cases}
    \]
Now consider $w_k:=\max(U-\frac{1}{k} v, U_k)$. Clearly, $w_k \nearrow U$ outside a pluripolar set, and $w_k \in {\mathcal{D}}(\Omega)$ since $w_k \geq U_k$. We will first show that $(dd^cw_k)^n\geq \mu$. Pick an arbitrary hyperconvex subdomain $\Omega' \Subset \Omega$, and pick a constant $C$ such that 
\[U - C, U_k -C, w_k -C \in \mathcal{E}(\Omega').
\]
Then $\max(U-C-\frac{1}{k} \max(v,-j), U_k-C) \searrow w_k-C$, and  
\[
(dd^c\max(U-C-\frac{1}{k} \max(v,-j), U_k-C))^n \geq \mu \text{ on }\Omega'
\]
by Lemma~\ref{max}. By the defining properties of $\mathcal{E}(\Omega')$, this sequence of measures is weak*-convergent to $(dd^cw_k)^n$, and we infer that  $(dd^cw_k)^n \geq \mu$ on the entire domain since $\Omega'$ was arbitrary. 

In fact, $w_k = U_k$. To see this, note that since $v$ is strictly positive, 
\[K:=\{w_k >U_k\}\Subset \Omega.\]
Therefore
\[
w_k-M_k \geq U_k-M_k \geq w_k - M_k + \psi_K, 
\]
where $\psi_K \in \mathcal{F}(\Omega)$ is defined by $\psi_K = U_k-M_k$ on a precompact neighborhood of $K$. Again applying Lemma~\ref{jamfor}, $w_k = U_k$. A key consequence is that $U_k \nearrow U$, with the desired conclusion $(dd^cU)^n = \mu$.  

Now suppose that $U'$ is another solution, distinct from $U$. Then, since $U'\leq U$ and $U_k \nearrow U$, for $k$ large enough there is a constant $C>0$ such that $\{U' + C < U_k\}$ is nonempty. Using Lemma~\ref{Equasi} and Lemma~\ref{bpluripolar}, we may construct an element $v \in \mathcal{E}(\Omega)$ with negative singularities at all points of $E_\varphi$, quasibounding $U_k$. Consequently
\[
U_k + \varepsilon v < U' \text{ on } \partial \Omega
\]
for all $\varepsilon >0$. Pick $\varepsilon$ small enough such that
\[
E := \{U' + C < U_k + \varepsilon v\}
\]
is nonempty as well. Then, since $E\Subset \Omega,$ we can find a hyperconvex subdomain $\Omega' \Subset \Omega$ containing $E$, with 
\[
U_k + \varepsilon v < U' \text{ on } \partial \Omega'.
\]
Subtracting constants, we may assume that both sides of the inequality lie in $\mathcal{E}(\Omega')$. Now
\[
U'' := \max(U_k + \varepsilon v, U')
\]
satisfies $(dd^cU'')^n \geq \mu$ by Lemma~\ref{max}, with $\{U'' > U'\} \Subset \Omega'$. A third application of Lemma~\ref{jamfor} implies that $U''=U'$ on $\Omega'$, and  in particular $E$ is empty. We conclude that $U'=U$ on $\Omega$.
\end{proof}
\begin{remark}
We conjecture that the assumption that $\varphi$ has minorant in $\mathcal{E}(\Omega)$ could be replaced by having a minorant in $\mathcal{D}(\Omega)$. Either way, there are pluri-quasibounded boundary data where there are no such minorants. For example, on the unit sphere in $\C^2$, $\varphi(z_1,z_2)=-(-\log|z_1|)^\alpha$ is pluri-quasibounded for $0<\alpha<1$, but does not have a minorant in $\mathcal{D}(\Omega)$ unless $0<\alpha <1/2$ (see \cite{nilsson} and the references therein).
\end{remark}
\begin{remark}
Although we are not concerned with investigating the weakest possible assumptions on $\mu$ in this paper, we remark that it might be possible to consider measures with finite total mass that do not charge pluripolar sets, possibly given some additional assumptions. This is because for such measures, the Dirichlet problem is uniquely solvable for bounded, continuous boundary data $\varphi$ in the sense that there is a unique element $u$ in the Cegrell class $\mathcal{F}(\varphi)$ such that $(dd^cu)^n=\mu$ (see \cite[Definition~2.12, Theorem~6.1]{czyzbok}). However, this apriori only guarantees that the upper limits of the solution coincide with $\varphi$.  We should also remark that there exist compliant measures that do not have finite total mass, as illustrated by the subharmonic functions 
\[u_\alpha(z) := -(1-|z|^2)^\alpha, \quad 0<\alpha < 1
\]
on the unit disk (see \cite[Example~3.8]{guedj}).
\end{remark}
\begin{remark}
The formulation of the theorem allows for an arbitrary b-pluripolar set of points where $\varphi$ is continuous and bounded to be included in $E_\varphi$. We remark that it is possible to show that the negligible sets are precisely those that are b-pluripolar (i.e. b-pluripolarity of $E_\varphi$ is also necessary for the Dirichlet problem to have a unique solution), see the proof of Theorem~\ref{oluri_main}.
\end{remark}
\section{Propagation of boundary singularities}
Similar to the main results in \cite{nilsson,nilsson2}, it is possible to say something about the regularity of the solution  if one additionally requires that $\mu$ is \textit{continuously compliant}, i.e. the Dirichlet problem
\[
    \begin{cases}
    u \in \psh(\Omega)\cap L^\infty(\Omega)  \\
    (dd^cu)^n=\mu \\
        \lim_{\Omega \ni \zeta \rightarrow z_0 \in \bd \Omega } u(\zeta) = \phi(z_0), \quad \forall z_0 \in \bd \Omega
    \end{cases}
\]
has a unique solution which is continuous for every \(\phi \in C(\bd \Omega)\).

We introduce the following terminology.
\begin{definition}
For a b-pluripolar set $E \subset \partial \Omega$, we define the \textit{propagating singularity set} $ E^{\ast} \subset \bar \Omega$ of $E$ by
\[
 E^{\ast} := \bigcap_{u \in \mathscr{F}_{E}}\{u_*(z) = -\infty\},
\]
where
	\[
	\mathscr{F}_{E}:=\{u\in \psh(\Omega) \suchthat  u<0, u^*\mid_{E} = -\infty \}.
	\]
\end{definition}
Note that $\hat E \subset E^{\ast}$, where
\[
 \hat E := \bigcap_{u \in \mathscr{F}_{E}}\{u^*(z) = -\infty\}
\]
is the b-pluripolar hull of $E$. They will in general not coincide, as $E^{\ast}$ is closed (since $\{u_*(z) \neq -\infty\}$ is open). We are currently unable to determine if $E^\ast=\bar{\hat E}$ in full generality; we conjecture that this is the case, although it is known (confer \cite[Example~4.6]{ahag3}) that there exists $u\in \psh(\Omega)$ on any bounded domain $\Omega$ such that
\[
\{u_*(z) = -\infty\} \neq \overline{\{u^*(z) = -\infty\}}.
\]
However, as the following proposition shows, the two sets do coincide when arising from the singularity set of a boundary function.

\begin{proposition}
Let $\varphi: \partial \Omega \rightarrow \R\cup\{-\infty,\infty\}$ be a boundary function. Let $E_\varphi \subset \partial\Omega$ denote the set where $\varphi$ is discontinuous or equal to $\pm\infty$, and suppose that $E_\varphi$ is b-pluripolar. Then  $E_\varphi^\ast=\bar{\hat E}_\varphi$.
\end{proposition}
\begin{proof}
Define
\begin{align*}
E^\pm &:=\{\zeta \in E_\varphi \suchthat \limsup_{\partial \Omega \ni \zeta \rightarrow \eta}\varphi(\zeta) = \liminf_{\partial \Omega \ni\zeta \rightarrow \eta}\varphi(\zeta)\} \\
 E^\text{disc} &:= E_\varphi \setminus E^\pm.
\end{align*}
As $\varphi$ is a mapping between two metrizable spaces, $E^\text{disc}$ must be a $F_\sigma$ set. Furthermore, $\overline{E^\pm} \subset E_\varphi$, so
\[
E_\varphi = E^\text{disc} \cup \overline{E^\pm},
\]
and we conclude that $E_\varphi$ must also be a $F_\sigma$ set. Therefore, we may find a sequence $E_1 \subset E_2 \subset \dots \subset E_\varphi$ of compact sets $E_k$ such that $\bigcup E_k = E_\varphi$. 

Now note that the statement is clearly true if $\widebar{\hat E}_\varphi = \widebar \Omega$, so suppose that the inclusion is strict. Then, given any point $z_0 \in \widebar \Omega\setminus \widebar{\hat E}_\varphi$, we may find a small closed ball $\bar B\subset \widebar \Omega\setminus \widebar{\hat E}_\varphi$, centered at $z_0$, and by \cite[Proposition~4.2]{djire}, there exist negative plurisubharmonic functions $u_k$ such that
\[
u_k \geq -\frac{1}{2^k} \text{ on }\bar B, \quad \limsup_{z \rightarrow E_k}u_k =-1.
\]
Using Harnack's theorem, it follows that the series
\[
\tilde u := \sum_{k=1}^\infty u_k
\]
must be an element in $\mathscr{F}_{E_\varphi}$, with $\tilde u_*(z_0) \geq -1$. Since $z_0$ was arbitrary, we conclude that $E^{\ast} \subset \widebar{\hat E}$, which finishes the proof.
\end{proof}
We are now ready to prove
\begin{theorem}\label{konti}
If $\mu$ is continuously compliant, then the unique pluri-quasibounded solution to the Dirichlet problem is continuous on $\Omega \setminus E^\ast_\varphi$. 
\end{theorem}

\begin{proof}
Assume first that $\varphi$ is bounded from above. Denote the unique solution by $u_\varphi$ and pick an arbitrary open set $\Omega' \Subset \Omega \setminus E_\varphi^*$. Since $\Omega'$ and $E_\varphi^*$ are separated, we may find $u' \in \mathcal{E}(\Omega)$, $u' < -C$, which quasibounds $u_\varphi$ from below with $u'>-C'$ on $\Omega'$. For each $k \in \N$, we may then find $C_k\geq 0$ such that
\[
\frac {u'}{k} - C_k < u_\varphi.
\]
Now let $u_k$ denote the unique solution to 
\[
    \begin{cases}
    u_k \in \psh(\Omega)\cap L^\infty(\Omega)  \\
    (dd^cu_k)^n=\mu \\
        \lim_{w \rightarrow z}u_k(w) = \max(\varphi(z), -C_k) \qquad \forall z \in \partial \Omega \setminus  E_\varphi.
    \end{cases}
\]
By \cite[Theorem~2.3]{nilsson2}, $u_k$ is continuous on $\Omega \setminus E_\varphi^\ast$. Using Lemma~\ref{jamfor}, $u_k + \frac {u'}{k}\leq u_\varphi$. But then
\[
u_k - \frac {C'}{k}  \leq u_k + \frac {u'}{k}\leq u_\varphi \leq u_k
\]
on $\Omega'$, so $u_k \searrow u_\varphi$ uniformly on $\Omega'$. Since $\Omega'$ was arbitrary, we conclude that $u_\varphi \in C(\Omega \setminus E_\varphi^\ast)$.

For the unbounded case, it is enough to show that the unique solution $U$ is lower semicontinuous on $\Omega'$. Let $v$ be a plurisuperharmonic function that quasibounds $U$, such that $v > C$ on $\Omega'$. Then we can find constants such that
\[
U<\frac{1}{k}v +M_k,
\]
and using the proof of Theorem~\ref{mainresult}, each $U_k$, defined as the quasibounded solutions to
    \[
    \begin{cases}
    U_k \in \mathcal{D}(\Omega) \\
    (dd^cU_k)^n=\mu \\
        \lim_{w \rightarrow z}U_k(w) = \min(\varphi(z), M_k) \qquad \forall z \in \partial \Omega \setminus  E_\varphi,
    \end{cases}
    \]
coincide with $w_k=\max(U-\frac{1}{k} v, U_k)$. Therefore, $U_k \nearrow U$ outside the singularity set of $v$, and in particular on $\Omega'$. By the previous paragraph, $U_k$ is continuous on $\Omega'$, which implies that $U$ is lower semicontinuous on $\Omega'$.
\end{proof}
\begin{remark}
As \cite[Example~2.4]{nilsson2} shows, this theorem is sharp in the sense that there exist boundary data $\varphi$ for which the discontinuity set of $u_\varphi$ is precisely $E_\varphi^*$. It is currently unknown to the author whether this is always the case.
\end{remark}
\section{Maximal pluri-quasibounded solutions}
For the homogeneous case, one may replace $(dd^cu)^n=0$ with instead requiring that $u$ is maximal (the two notions coincide if $u\in \mathcal{D}(\Omega)$, see \cite{blocki2}). We shall use the following two characterizations of maximality, found in \cite{sadu}: $u\in \psh(\Omega)$ is maximal if and only if for any $u' \in \psh(\Omega)$,
\[
\{u'>u\} \Subset \Omega \implies \{u'>u\} = \emptyset,
\] 
or equivalently, for each open subset $\Omega' \Subset\Omega$,
\[
u\geq u' \text{ on }\partial \Omega' \implies u\geq u' \text{ on }\Omega'.
\]
This circumvents the need to have well-defined complex Monge--Ampère measures, and the assumption that there exists an element in  $\mathcal{E}(\Omega)$ minorizing the boundary data may be dropped. In complete analogy with Corollary~\ref{dirichlet_flerdim}, we then have the following theorem.

\begin{theorem}\label{oluri_main}
    Suppose that $\varphi: \partial \Omega \rightarrow \R\cup\{-\infty, \infty\}$ is pluri-quasibounded, continuous and locally bounded outside $E_\varphi$. Then there exists a unique pluri-quasibounded, maximal plurisubharmonic function $u$ such that
    \[
        \lim_{w \rightarrow z}u(w) = \varphi(z) \qquad \forall z \in \partial \Omega \setminus  E_\varphi
    \]
    if and only if $E_\varphi$ is b-pluripolar. Furthermore, $u$ is continuous on $\Omega \setminus E^\ast_\varphi$.
\end{theorem} 
\begin{proof}
Sufficiency, as well as the continuity statement, can be inferred from the proofs of Theorem~\ref{mainresult} and Theorem~\ref{konti}, replacing the uses of Lemma~\ref{jamfor} with the above maximality property. The only substantial difference is to show that the solution candidate
\[
U := \sup\{u(z) \in  \psh(\Omega) \suchthat u \leq h_\varphi\}
\]
is maximal, as maximality is not necessarily preserved on increasing sequences. But this follows from the following standard trick: Suppose that $u'\in \psh(\Omega)$ satisfies $U\geq u'$ on $\partial \Omega'$, with $\Omega' \Subset \Omega$. Then
\[
        u'' := \begin{cases}
            U, & z \in \Omega \setminus \Omega' \\
            \max(u', U), & z \in \Omega'
        \end{cases}
    \]
    belongs to the defining envelope of $U$, and we conclude that $u' \le U$ on $\Omega'$, i.e. $U$ is maximal. 

It remains to show that $E_\varphi$ is necessarily b-pluripolar. Denote the unique solution by $u_\varphi$, and suppose without loss of generality that $v\geq |u_\varphi|$  is a positive plurisuperharmonic function that quasibounds $u_\varphi$. 
Decompose 
\[
E_\varphi= E^+ \cup E^- \cup E^s,
\]
where
\begin{align*}
    E^+&:=\{\zeta \in E_\varphi \suchthat \liminf_{z\rightarrow \zeta}u_\varphi =\infty\}\\
    E^-&:=\{\zeta \in E_\varphi \suchthat \limsup_{z\rightarrow \zeta}u_\varphi =-\infty\} \\
    E^s&:=E_\varphi\setminus E^+ \cup E^-.
\end{align*}
Clearly $E^+\cup E^-$ is b-pluripolar, since $|u_\varphi|\leq v$. Suppose for a contradiction that $E^s$ is not b-pluripolar. Define for each $k$
\[
\mathscr{F}_k:= \{u \in \psh(\Omega) \suchthat u \leq v \text{ and } u^* \leq \max(\varphi, -k) \text{ on }\partial\Omega \setminus E^{s}\}
\]
and consider the envelope
\[
u_{L,k} := \sup\{u(x) \suchthat u\in \mathscr{F}_k,  u^* \leq -L \text{ on } E^{s}\}
\]
for $L>0$. By the Brelot--Cartan theorem \cite[Theorem~4.42]{guedj}, $u_{L,k}^*$ is a plurisubharmonic function, which clearly attains the correct boundary limits on $\partial\Omega\setminus \big(E^s\cup\{\varphi < k\} \big)$. We shall first prove that $u_{L,k}^*$ is maximal as follows. Using Choquet's lemma \cite[Lemma~4.31]{guedj}, there exist $u_1,u_2, \ldots$ in the defining envelope of $u_{L,k}$ such that
\[
(\sup_j u_j)^* = u_{L,k}^*.
\]
Replacing $u_j$ with 
\[
\sup\{u(z)  \suchthat u \in \psh(\Omega), u^* \leq \max(-k-L, u_1, \dots, u_{j-1}) \text{ on }\partial \Omega\},
\]
we may assume that the sequence $(u_j)$ is increasing, with each $u_j$ maximal and bounded from below. Since $u_j \in L^\infty_{loc}(\Omega)$, maximality implies that $(dd^cu_j)^n = 0$. By monotonicity of the complex Monge--Ampère operator, we conclude that
\[
(dd^cu_{L,k}^*)^n = 0,
\]
which shows that $u_{L,k}^*$ is maximal. 

We will now show that for $L$ large enough, there exists $z_0 \in \Omega$ such that $u_{L,k}^*(z_0) <u_\varphi(z_0)$. Indeed, for any function $u'$ in the defining envelope,
\[
u' - v \leq 0, \quad \limsup_{z\rightarrow E^s}(u'-v) \leq -L,
\]
which implies that
\[
u' \leq L \cdot \omega^*(z,E^{s},\Omega) + v,
\]
where 
\[
\omega^*(z,E^{s},\Omega) = \big(\sup\{u(x) \suchthat u\in \psh^-(\Omega), u \leq -1 \text{ on }E^{{s}}\}\big)^*
\]
denotes the boundary relative extremal function. Since $E^s$ is not b-pluripolar, by \cite[Proposition~3.5]{djire}, $\omega^*(z,E^{s},\Omega)$ is not identically zero. Therefore, for any $\Omega'\Subset \Omega$, there is a constant $C>0$ such that 
\[
-1< \omega^*(z,E^{s},\Omega) < -C \text{ on } \Omega',
\]
and we conclude that $u_{L,k}^* \leq -LC + v$ on $\Omega'$. Furthermore, since $\Omega'$ is open, there exists $z_0 \in \Omega'$ such that 
 $v(z_0)<\infty$, and fixing $L$ with
\[
L > \frac{v(z_0)-u_\varphi(z_0)}{C}
\]
yields $u_{L,k}^*(z_0) <u_\varphi(z_0)$.

To finish the proof, notice that for each $\varepsilon>0$, there exists $M_\varepsilon$ such that
\[
 |u_{L,k}^*| \leq \varepsilon v + M_\varepsilon-L\omega^*(z,E^{s},\Omega) \leq \varepsilon v + {M_\varepsilon + L}
\]
for all $k<0$. Therefore, letting $k \rightarrow - \infty$, the sequence $(u_{L,k}^*)$ decreases to a pluri-quasibounded, plurisubharmonic function $U_L$, which is maximal since maximality is preserved on decreasing sequences. It is also clear that $U_L$ satisfies
\[
        \lim_{w \rightarrow z}u_L(w) = \varphi(z) \qquad \forall z \in \partial \Omega \setminus  E_\varphi,
\]
so $U_L$ is another solution, by construction distinct from $u_\varphi$.
\end{proof}
\begin{remark}
As mentioned in the third remark to Theorem~\ref{mainresult}, small adjustments to the above proof also yield necessity of having $E_\varphi$ b-pluripolar for the inhomogeneous case. 
\end{remark}


\begin{thebibliography}{99}

    \bibitem{ahag2}
    P. Åhag, U. Cegrell, R. Czyż, P.H. Hiep, \emph{Monge--Ampère measures on pluripolar sets}, J. Math. Pures Appl. 92, pp.~613--627 (2009).

    \bibitem{ahag}
    P. Åhag, R. Czyż, C.H. Lu, A. Rashkovskii, \emph{Geodesic connectivity and rooftop envelopes in the Cegrell classes}, Math. Ann. 391 (2025), pp.~3333--3361.
    
    \bibitem{ahag3}
    P. Åhag, U. Cegrell, P.H. Hiep, \textit{Monge--Ampère measures on subvarieties}, Journal of Mathematical Analysis and Applications, \textbf{423}, 1 (2015), pp.~94--105.

    \bibitem{arsove}
    M. Arsove, H. Leutwiler, \textit{Quasi-bounded and singular functions}, Trans.
    Amer. Math. Soc. \textbf{189} (1974), pp.~275--302.

    \bibitem{blocki}
    Z. Błocki, \emph{The Domain of Definition of the Complex Monge--Ampère Operator}, Amer. Math. J., vol. 128, no. 2 (2006), pp.~519--30.

    \bibitem{blocki2}
    Z. Błocki, \emph{On the definition of the Monge--Ampère operator in $\C^2$}, Math. Ann. 328 (2004), pp.~415--423.

  \bibitem{cegrell} 
  U. Cegrell, \emph{Pluricomplex energy},
  Acta Math., \textbf{180} (1998), pp.~187--217.

  \bibitem{cegrell2} 
  U. Cegrell, \emph{The general definition of the complex {M}onge-{A}mp\`ere
              operator}, Ann. Inst. Fourier (Grenoble), \textbf{54} (2004), pp.~159--179.

    \bibitem{czyzbok}
    R. Czyż, \textit{The complex Monge--Ampère operator in the Cegrell classes}, Dissertationes Math. \textbf{466}, 83 (2009).

    \bibitem{czyz} R. Czyż, \emph{On a {M}onge-{A}mp\`ere type equation in the Cegrell class $\mathcal{E}_{\chi}$}, Annales Polonici Mathematici 99.1 (2010), pp.~89--97.

    \bibitem{djire}
    I. K. Djire, J. Wiegerinck, \emph{Characterizations of boundary pluripolar hulls}, 
    Complex Variables and Elliptic Equations, 61(8) (2016), pp.~1133--1144.

    \bibitem{guedj}
    V. Guedj, A. Zeriahi, \textit{Degenerate Complex Monge--Ampère Equations}, EMS (2017).

    
 

    \bibitem{nilsson}
    M. Nilsson, \textit{Continuity of envelopes of unbounded plurisubharmonic functions}, Math. Z. \textbf{301} (2022), pp.~3959--3971.

    \bibitem{nilsson2}
    M. Nilsson, \textit{Plurisubharmonic functions with discontinuous boundary behavior}, Indiana Univ. Math. J. Vol. 74, No. 2 (2025). 

    \bibitem{nilsson3}
    M. Nilsson, F. Wikström, \emph{Quasibounded plurisubharmonic functions}, Internat. J. Math., Vol. 32, No. 9 (2021).

    \bibitem{parreau}
    M. Parreau, \emph{Sur les moyennes des fonctions harmoniques et analytiques
    et la classification des surfaces de Riemann}, Ann. Inst. Fourier, 3 (1951),
    pp.~103--197.

    \bibitem{ransford}
    T. Ransford, \textit{Potential Theory in the Complex Plane}, London Math. Soc. Student Texts No. 28 (Cambridge University Press, 1995).

    \bibitem{rashkovskii}
    A. Rashkovskii, \emph{Rooftop envelopes and residual plurisubharmonic functions}, Annales Polonici Mathematici 128 (2022), pp.~159--191.
    
    \bibitem{hiep}
    N. Van Khue, P.H. Hiep,  \emph{A Comparison Principle for the Complex Monge--Ampère Operator in Cegrell’s Classes and Applications}, Transactions of the American Mathematical Society \textbf{361} (10) (2009), pp.~5539--5554.

    \bibitem{sadu}
    A. Sadullaev, \emph{Plurisubharmonic measures and capacities on complex manifolds}, Russ. Math. Surv. 36, No. 4 (1981), pp.~61--119.

    \bibitem{xing}
    Y. Xing, \emph{Complex Monge--Ampère measures of plurisubharmonic functions with bounded values near the boundary}, Canad. J. Math. Vol. \textbf{52} (5) (2000), pp.~1085–1100.
    
    \bibitem{yamashita}
    S. Yamashita, \emph{On some families of analytic functions on Riemann Surfaces}, Nagoya Math. J., 31 (1968), pp.~57--68.
\end{thebibliography}
\end{document}